\documentclass[10pt,psamsfonts]{amsart}
\usepackage{amssymb,accents}
\usepackage{amsmath}
\usepackage{graphicx}
\usepackage{amscd}
\usepackage{amsfonts}
\usepackage{amsbsy}
\usepackage[T1]{fontenc}
\usepackage[english]{babel}

\usepackage{enumerate}

\usepackage{centernot}
\usepackage{mathtools}

\newtheorem{theorem}{Theorem}[section]
\newtheorem{prop}[theorem]{Proposition}
\newtheorem{lemma}[theorem]{Lemma}
\newtheorem{corollary}[theorem]{Corollary}
\newtheorem{open}[theorem]{Open question}

\newtheorem*{problem*}{Problem}
\newtheorem{hypo}[theorem]{IFS conditions}
\newtheorem*{hypo*}{IFS conditions}

\newtheorem*{moran*}{Moran's Theorem (1946)}

\newtheorem*{matti*}{Mattila's Counterexample}

\newtheorem*{schief*}{Schief's Counterexample}

\theoremstyle{definition}
\newtheorem{definition}[theorem]{Definition}

\theoremstyle{remark}
\newtheorem{remark}[theorem]{Remark}

\numberwithin{equation}{section}

\newcommand{\R}{\mathbb R}

\newcommand{\N}{\mathbb N}

\newcommand{\TT}{{\mathbb T}^{2}}

\newcommand{\SSS}{{\mathbb S}}

\def \ord{{\rm Ord \,}}

\def \inf{{\rm inf}}

\def \dim{{\rm dim \,}}
\def \dm{{\rm diam \,}}

\def \cd{{\rm Card \,}}

\def \SOSC{{\rm SOSC}}
\def \OSC{{\rm OSC}}
\def \WOSC{{\rm WSC}}
\def \WSP{{\rm LSP}}
\def \LSP{{\rm LSP}}
\def \LSPuno{{\rm LSP1}}
\def \LSPdos{{\rm LSP2}}
\def \IFS{{\rm IFS}}
\def \EIFS{{\rm EIFS}}

\def \K{{\mathcal{K}}}

\newcommand{\sqs}{\sqsubseteq}

\newcommand \F{\mathcal{F}}

\newcommand \St{{\rm St}}
\newcommand \Fr{{\rm Fr}}
\newcommand \ef{\mathbf{\Gamma}}

\newcommand \dih{\dim_{\rm H}}

\newcommand \tres{\dim_{\ef}^{3}}

\newcommand \cuatro{\dim_{\ef}^{4}}

\newcommand \h{\dim_{\rm H}}

\newcommand \eps{\varepsilon}

\newcommand \deltafn{\dm(F,\Gamma_n)}

\newcommand \inte{\accentset{\circ}}
\newcommand \anf{\mathcal{A}_n(F)}
\newcommand \ankf{\mathcal{A}_{n,k}(F)}
\newcommand \hks{\mathcal{H}_{k}^s(F)}

\begin{document}

\title[Irreducible fractal structures for Moran's type theorems]
{Irreducible fractal structures for Moran's type theorems}

\author[M.A. S\'anchez-Granero and M. Fern\'andez-Mart\'{\i}nez]{M.A. S\'anchez-Granero$^1$ and M. Fern\'andez-Mart\'{\i}nez$^{2}$}

\address{$^{1}$ Departamento de Matem\'aticas, Universidad de Almer\'{\i}a, 04120 La Ca\~nada de San Urbano, Almer\'{\i}a (SPAIN)}

\email{misanche@ual.es}

\address{$^{2}$ University Centre of Defence at the Spanish Air Force Academy, MDE-UPCT, 30720 Santiago de la Ribera, Murcia (SPAIN)}
\email{fmm124@gmail.com}



\thanks{
\newline \indent 
Both authors are partially supported by grant No.~MTM2015-64373-P (MINE\-CO/FEDER, UE).
The second author also acknowledges the partial support of grant No.~MTM2014-51891-P from Spanish Ministry of Economy and Competitiveness and grant No.~19219/PI/14 from Fundaci\'on S\'eneca of Regi\'on de Murcia. 
}


\begin{abstract}
In this paper, we characterize a novel separation property for IFS-attractors on complete metric spaces. Such a separation property is weaker than the strong open set condition ($\SOSC$) and becomes necessary to reach the equality between the similarity and the Hausdorff dimensions of strict self-similar sets. 
We also investigate the size of the overlaps from the viewpoint of that separation property. 
In addition, we contribute some equivalent conditions to reach the equality between the similarity dimension and a new Hausdorff type dimension for IFS-attractors introduced by the authors in terms of finite coverings.
\end{abstract}

\maketitle

\section{Introduction}

A classical problem in Fractal Geometry consists in determining under what conditions on the pieces of a strict self-similar set $\K$, the equality between the similarity and the Hausdorff dimensions of $\K$ holds. A classical result contributed by Australian mathematician P.A.P. Moran in the forties (c.f.~\cite[Theorem III]{Moran1946}) states that under the open set condition ($\OSC$), which is a property required to the pieces of $\K$ to guarantee that their overlaps are thin enough, the desired equality stands. Afterwards, Lalley introduced the strong open set condition ($\SOSC$) by further requiring that the (feasible) open set provided by the $\OSC$ must intersect the attractor $\K$. It is worth pointing out that the next chain of implications and equivalences stands in the case of Euclidean self-similar sets and is best possible (c.f.~\cite{Schief1994}):
\begin{equation}\label{eq:schief94}
\SOSC\Leftrightarrow \OSC\Leftrightarrow \mathcal{H}_{\rm H}^{\alpha}(\K)>0\Rightarrow \h(\K)=\alpha,	
\end{equation} 
where $\mathcal{H}_{\rm H}^{\alpha}$ is the $\alpha-$dimensional Hausdorff measure, $\h$ denotes the Hausdorff dimension, and $\alpha$ is the similarity dimension of the attractor. Interestingly, Schief proved that $\mathcal{H}_{\rm H}^{\alpha}(\K)>0\Rightarrow \SOSC$ (c.f.~\cite[Theorem 2.1]{Schief1994}) which implies that both the $\SOSC$ and the $\OSC$ are equivalent for Euclidean IFS-attractors.
A counterexample due to Mattila (c.f.~Section \ref{sec:4}) guarantees that the last implication in Eq.~(\ref{eq:schief94}) does not hold, in general. Accordingly, the $\OSC$ becomes only sufficient to reach the equality between those dimensions. A further extension of the problem above takes place in the more general context of attractors on complete metric spaces. Schief also explored such a problem and justified the following chain of implications (c.f.~\cite{Schief1996}):
\begin{equation}\label{eq:schief96}
\mathcal{H}_{\rm H}^{\alpha}(\K)>0\Rightarrow \SOSC\Rightarrow \h(\K)=\alpha,	
\end{equation} 
i.e., the $\SOSC$ is necessary for $\mathcal{H}_{\rm H}^{\alpha}(\K)>0$ and only sufficient for $\h(\K)=\alpha$. Once again, the above-mentioned result of Mattila implies that Eq.~(\ref{eq:schief96}) is best possible. From both Eqs.~(\ref{eq:schief94}) and (\ref{eq:schief96}), it holds that the $\SOSC$ is a sufficent condition on the pre-fractals of $\K$ leading to $\h(\K)=\alpha$. 

In this paper, we make use of the concept of a fractal structure (first contributed in \cite{bandt1992}) to explore and characterize a novel separation property in both contexts: Euclidean attractors and self-similar sets in complete metric spaces. Such a separation property, weaker than the $\OSC$, becomes necessary to reach the equality between the similarity dimension of the attractor and its Hausdorff dimension. Accordingly, we shall conclude (in the general case) that 
\begin{equation*}
\mathcal{H}_{\mathrm H}^{\alpha}(\K)>0\Rightarrow \SOSC\Rightarrow \h(\K)=\alpha\Rightarrow
\WOSC,	
\end{equation*}
where $\WOSC$ refers to the weak separation condition for attractors we shall introduce in upcoming Section~\ref{sec:wsc}.

Moreover, we will prove that the WSC holds if and only if $\cuatro(\K)=\alpha$, where $\cuatro$ is fractal dimension IV introduced in \cite{DIM4} (c.f.~Section~\ref{sec:3}).

\section{Preliminaries}

\subsection{General notation}

Along the sequel, we shall use the following notation from the domain of words \cite{Bandt2005,MR0426486,MR1467773}.
Let $\Sigma=\{1,\ldots,k\}$ be a finite (nonempty) set (also called an alphabet). 
For each natural number $n$, let $\Sigma^n=\{\mathbf{i}=i_1\cdots\ i_n: i_j\in \Sigma, j=1,\ldots, n\}$ be the set consisting of all the words of length $n$ from $\Sigma$. In addition, let $\Sigma^{\infty}$ denote the collection of either all finite $(\cup_{n\in \N} \Sigma^n)$ or infinite $(\Sigma^{\N})$ words from $\Sigma$, i.e., $\Sigma^{\infty}=\cup_{n\in \N}\Sigma^n \cup \Sigma^{\N}$. Thus, the prefix order $\sqs$ is defined on $\Sigma^{\infty}$ as follows: $x\sqs y$, if and only if, $x$ is a prefix of $y$.  

\subsection{$\IFS-$attractors}

Let $k\geq 2$. By an iterated function system ($\IFS$), we shall understand a finite collection of similitudes on a complete metric space $(X,\rho)$, say $\F=\{f_1,\ldots,f_k\}$, where each self-map $f_i:X\to X$ satisfies the following identity: 
\[\rho(f_i(x),f_i(y))=c_i\cdot \rho(x,y),\text{ for all }x,y\in X, \] 
with $c_i\in (0,1)$ being the similarity ratio associated with $f_i$.
In particular, if $X=\R^d$, then $\F$ is said to be an Euclidean $\IFS$ ($\EIFS$, hereafter). 
Under the previous assumptions, there always exists a unique (nonempty) compact subset $\K\subseteq X$ such that
\begin{equation}\label{eq:1}
\K=\cup \{f_i(\K):i=1,\ldots,k\}.	
\end{equation}
The previous equality is usually known as Hutchinson's equation \cite{HUT81} and $\K$ is said to be the $\IFS-$attractor (also the self-similar set) generated by $\F$. It is worth pointing out that $\K$ consists of smaller self-similar copies of itself, $\K_i$, named as pre-fractals of $\K$. Thus, $\K_i=f_i(\K)$ for all $i=1,\ldots,k$. In addition, we shall write $\K_{ij}=f_i(f_j(\K))$, and so on. 
Hence, if $f_{\mathbf{i}}=f_{i_{1}}\circ \cdots \circ f_{i_{n}}, c_{\mathbf{i}}=c_{i_{1}}\cdots\ c_{i_{n}}$, and $\K_{\mathbf{i}}=f_{\mathbf{i}}(\K)$, then Eq.~(\ref{eq:1}) can be rewritten in the following terms:
\[\K=\cup\{\K_{\mathbf{i}}:\mathbf{i}\in \Sigma^n\}.\]
In addition, the address map $\pi:\Sigma^{\infty}\to \K$ stands as a continuous map from the collection $\Sigma^{\infty}$ of all the words of infinite length (sequences) onto the $\IFS-$attractor~$\K$. It is worth pointing out that if the similarity ratios $c_i$ are \emph{small}, then the pre-fractals $\K_i$ are disjoint, $\pi$ is a homemomorphism, and $\K$ becomes a Cantor set. 



\subsection{The open set condition}
In the Euclidean case, there are, at least, three equivalent descriptions regarding the open set condition (OSC in the sequel), which controls the overlaps among the pre-fractals of $\K$.

\begin{enumerate}[(i)]
\item \textbf{The Moran's open set condition} (due to P.A.P. Moran, c.f.~\cite{Moran1946}). We say that $\F=\{f_1,\ldots,f_k\}$ (or its attractor $\K$) is under the OSC if there exists a nonempty open subset $\mathcal{V}\subseteq \R^d$ such that the images $f_i(\mathcal{V})$ are pairwise disjoint with all of them being contained in $\mathcal{V}$, which is called 
a \emph{feasible open set} of $\F$ 
(resp., of $\K$).  
\item \textbf{The finite clustering property} (contributed by Schief, c.f.~\cite{Schief1994}). There exists an integer $N$ such that at most $N$ incomparable pieces $\K_{\mathbf{j}}$ of size $\geq \eps$ can intersect the $\eps-$neighborhood of a piece $\K_{\mathbf{i}}$ of diameter equal to $\eps$. It has to be mentioned here that two pieces of $\K$, $\K_{\mathbf{j}}$ and $\K_{\mathbf{k}}$, are said to be incomparable if 
$\mathbf{j}\not\sqsubseteq \mathbf{k}$ and
$\mathbf{k}\not\sqsubseteq \mathbf{j}$.

\item \textbf{Positive $\alpha-$dimensional Hausdorff measure} (c.f.~\cite{Moran1946,Schief1994}): $\mathcal{H}_{\rm H}^{\alpha}(\K)>0$, where $\alpha$ is the similarity dimension of $\K$, i.e., the (unique) solution of the equation $\sum_{i=1}^kc_i^{\alpha}=1$ (c.f.~Definition \ref{def:sdim}).

\end{enumerate}
Lalley strengthened the $\OSC$ since the feasible open set $\mathcal{V}$ and the attractor $\K$ may be disjoint. Thus, the classical $\OSC$ may be too weak in order to obtain results regarding the fractal dimension of $\K$. In this way,
the strong open set condition ($\SOSC$) stands, if and only if, it holds, in addition to the $\OSC$, that $\K\cap \mathcal{V}\neq \emptyset$ (c.f.~\cite{Lalley1988}). Schief proved that both the $\OSC$ and the $\SOSC$ are equivalent on Euclidean spaces (c.f.~\cite[Theorem 2.2]{Schief1994}). Such a result has been further extended to conformal IFSs \cite{Peres2001}, and self-conformal random fractals \cite{Patzschke2003}, as well.

Finally, we should mention here that Schief has already explored some conditions to guarantee the equality between the similarity dimension (c.f.~Definition \ref{def:sdim}) and the Hausdorff dimension of $\IFS-$attractors on complete metric spaces. In this case, though, the $\OSC$ no longer leads to the equality between such fractal dimensions (c.f.~\cite{Schief1996}).

\subsection{Fractal structures}

Fractal structures were first sketched by Bandt and Retta in \cite{bandt1992} and introduced and applied afterwards by Arenas and S\'anchez-Granero in \cite{SG99A} to characterize non-Archimedean quasi-metrization. 
By a covering of a nonempty set $X$, we shall understand a family $\Gamma$ of subsets such that $X=\cup\{A:A\in \Gamma\}$. Let $\Gamma_1$ and $\Gamma_2$ be two coverings of $X$. The notation $\Gamma_2\prec \Gamma_1$ means that $\Gamma_2$ is a \emph{refinement} of $\Gamma_1$, namely, for all $A\in \Gamma_2$, there exists $B\in \Gamma_1$ such that $A\subseteq B$. Moreover, $\Gamma_2\prec\prec \Gamma_1$ denotes that $\Gamma_2\prec \Gamma_1$, and additionally, for all $B\in \Gamma_1$, it holds that $B=\cup\{A\in \Gamma_2:A\subseteq B\}$. Thus, a fractal structure on $X$ is a countable family of coverings $\ef=\{\Gamma_n\}_{n\in \N}$ such that $\Gamma_{n+1}\prec\prec \Gamma_n$, for all natural number $n$. The covering $\Gamma_n$ is called \emph{level} $n$ of $\ef$. It is worth mentioning that a fractal structure induces a transitive base of quasi-uniformity (and hence, a topology) given by the transitive family of entourages 
$U_{\Gamma_n}=\{(x,y)\in X\times X: y\in X\setminus\cup_{A\in \Gamma_n, x\notin A}A\}$.
Along the sequel, we shall allow that a set could appear twice or more in any level of a fractal structure. 
Let $\ef$ be a fractal structure on $X$ and assume that $\St(x,\ef)=\{\St(x,\Gamma_n)\}_{n\in \N}$ is a neighborhood base for all $x\in X$, where $\St(x,\Gamma_n)=\cup\{A\in \Gamma_n:x\in A\}$. Then $\ef$ is called a starbase fractal structure.
A fractal structure is said to be finite if all its levels are finite coverings. An example of a finite fractal structure is the one which any $\IFS-$attractor can be always endowed with naturally. Such a fractal structure plays a key role in this paper so we shall formally define it next.
\begin{definition}[c.f.~\cite{Arenas2011}, Definition 4.4]\label{def:nfsifs}
Let $\mathcal{F}$ be an $\IFS$ whose attractor is $\K$. The natural fractal structure on $\K$ as a self-similar set is given by the countable family of coverings $\ef=\{\Gamma_n\}_{n\in \N}$, where $\Gamma_n=\{f_{\mathbf{i}}(\K):\mathbf{i}\in \Sigma^n\}$.
\end{definition}
Alternatively, the levels of the natural fractal structure for any $\IFS-$attractor $\K$ can be described as $\Gamma_1=\{f_i(\K):i\in \Sigma\}$, and $\Gamma_{n+1}=\{f_i(A):A\in \Gamma_n,i\in \Sigma\}$ for every $n\in \N$.
It is also worth mentioning that such a natural fractal structure is starbase (c.f.~\cite[Theorem 4.7]{Arenas2011}). In addition, the next remark will result useful for upcoming purposes.
\begin{remark}\label{obs:incomp}
All the elements in a same level $\Gamma_n$ are incomparable.
\end{remark}





\section{Fractal dimensions for fractal structures} \label{sec:3}

The fractal dimension models for a fractal structure involved along this paper, i.e., fractal dimensions III and IV, have been explored in previous works by the authors (c.f.~\cite{DIM3,DIM4}) and can be considered as subsequent models from those studied in~\cite{DIM1}. It is worth pointing out that they allowed to generalize both box dimension (c.f.~\cite[Theorem 4.15]{DIM3}) and Hausdorff dimension (c.f.~\cite[Theorem 3.13]{DIM4}) in the context of Euclidean sets endowed with their natural fractal structures (c.f.~\cite[Definition 3.1]{DIM1}). Thus, they become ideal candidates to explore the self-similar structure of $\IFS-$attractors. Next, we provide their definitions. 

Let $\ef$ be a fractal structure on a metric space $(X,\rho)$. We shall define $\anf$ as the collection consisting of all the elements in level $n$ of $\ef$ that intersect a subset $F$ of $X$. Mathematically, $\anf=\{A\in \Gamma_n:A\cap F\neq \emptyset\}$. Further, let $\dm(\Gamma_n)=\sup\{\dm(A):A\in \Gamma_n\}$, and $\dm(F,\Gamma_n)=\sup\{\dm(A):A\in \anf\}$, as well. 
\begin{definition}(c.f.~\cite[Definition 4.2]{DIM3} and \cite[Definition 3.2]{DIM4})\label{def:1}
Assume that $\dm(F,\Gamma_n)\to 0$ and consider the following expression, where $k=3,4$:
\[\mathcal{H}_{n,k}^s(F)=\inf\left\{\sum \dm(A_i)^s:\{A_i\}_{i\in I}\in \ankf\right\}, \text{with}\]
\begin{enumerate}[(i)]
\item $\mathcal{A}_{n,3}(F)=\{\mathcal{A}_l(F):l\geq n\}$.\label{dim:3}
\item $\mathcal{A}_{n,4}(F)=\{\{A_i\}_{i\in I}:A_i\in \cup_{l\geq n}\Gamma_l,F\subseteq \cup_{i\in I}A_i,\cd(I)<\infty\}$. Here, $\cd(I)$ denotes the cardinal number of $I$.
\end{enumerate}
In addition, let $\hks=\lim_{n\to \infty}\mathcal{H}_{n,k}^s(F)$. By the fractal dimension III (resp., IV) of $F$, we shall understand the (unique) critical point satisfying the identity
\[\dim_{\ef}^k(F)=\sup\{s\geq 0:\hks=\infty\}=\inf\{s\geq 0:\hks=0\}.\]
\end{definition}
It is worth pointing out that fractal dimension III always exists since the sequence $\{\mathcal{H}_{n,3}^s(F)\}_{n\in \N}$ is monotonic in $n\in \N$. 
Likewise, the hypothesis $\dm(F,\Gamma_n)\to 0$, though necessary in such a definition, is not too restrictive as the following remark points out. 
\begin{remark}\label{obs:1}
Let $\K$ be an $\IFS-$attractor (with $\ef$ the natural fractal structure as a self similar set). Then it holds that $\dm(\K,\Gamma_n)\to 0$, since the sequence of diameters $\{\dm(\Gamma_n)\}_{n\in \N}$ decreases geometrically.
\end{remark}
 
\section{Some Moran's type theorems under the $\OSC$}\label{sec:4}

One of the main goals in this paper is to explore some separation conditions for $\IFS-$attractors in the context of fractal structures. To deal with, we shall endow each attractor with its natural fractal structure as self-similar set (c.f.~either Definition \ref{def:nfsifs} or Remark \ref{obs:1}). Next, we collect several assumptions regarding the $\IFS-$attractors involved along this paper. They are stated in the general context of complete metric spaces.

\begin{hypo}\label{hypo:1}
Let $(X,\F)$ be an $\IFS$, where $X$ is a complete metric space, $\F=\{f_1,\ldots,f_k\}$ is a finite collection of similitudes on $X$, and $\K$ is the $\IFS-$attractor of $\F$.  
In addition, let $\ef$ be the natural fractal structure on $\K$ as a self-similar set (c.f.~Definition \ref{def:nfsifs}), and $c_i$ be the similarity ratio of each $f_i\in \F$. 
\end{hypo}
If $(X,\F)$ satisfies $\IFS$ conditions \ref{hypo:1}, then we shall say, for short, that $\F$ is under $\IFS$ conditions \ref{hypo:1}. 
It is worth mentioning that all the results provided in this paper stand under $\IFS$ conditions~\ref{hypo:1}.

Next, we recall the concept of similarity dimension for $\IFS-$attractors. 
\begin{definition}\label{def:sdim}
Let $\F$ be an $\IFS$ and $\K$ its attractor. 
By the similarity dimension of $\K$, we shall understand the unique solution $\alpha>0$ of the equation $\sum_{i=1}^k c_i^s=1$. In other words, the similarity dimension of $\K$ is the unique value $\alpha>0$ such that $P(\alpha)=0$, where $P(s)=\sum_{i=1}^{k} c_i^{s}-1$.
\end{definition}
Along the sequel, $\alpha$ will denote the similarity dimension of an $\IFS-$attractor. It is worth noting that (without any additional assumption) $\mathcal{H}_{\rm H}^{\alpha}(\K)<\infty$ for any IFS-attractor $\K$ (c.f.~\cite[Proposition 4 (i)]{HUT81}).

The two results that follow are especially useful for upcoming purposes. The first one states that fractal dimension III (c.f.~Definition \ref{def:1} (\ref{dim:3})) equals the similarity dimension of $\IFS-$attractors without requiring the $\IFS$ $\F$ (resp., the attractor $\K$) any additional separation property. On the other hand, we also recall the classical Moran's Theorem, a standard result that gives the equality between both the Hausdorff and the similarity dimensions of $\IFS-$attractors lying under the $\OSC$.

\noindent
\textbf{A note to readers.} Each theoretical result provided along this paper has been assigned one of the two following labels: $\IFS$ or $\EIFS$. In the first case, it helps the reader to keep in mind that the corresponding result stands for attractors on complete metric spaces, whereas the label $\EIFS$ means that the result holds for Euclidean $\IFS-$attractors.

\begin{theorem}[IFS] (c.f.~\cite[Theorem 4.20]{DIM3})\label{teo:OSC-H=3}
$\tres(\K)=\alpha$ and $0<\mathcal{H}_3^{\alpha}(\K)<\infty$.
\end{theorem}

\begin{moran*}[$\EIFS$]
$\OSC\Rightarrow \h(\K)=\alpha$ and $0<\mathcal{H}_{\rm H}^{\alpha}(\K)<\infty$. 
\end{moran*}
By a Moran's type theorem, we shall understand a result that yields the equality between a fractal dimension $\dim$ of an $\IFS-$attractor $\K$ and its similarity dimension, i.e., $\dim(\K)=\alpha$. 
Theorem \ref{teo:OSC-H=3} and Moran's Theorem give the following result involving the fractal dimension III of $\K$ in the Euclidean case.

\begin{corollary}[$\EIFS$](c.f.~\cite[Corollary 4.22]{DIM3}) 
\label{cor:H=3=B}
$\OSC\Rightarrow \dih(\K)=\tres(\K)=~\alpha$.
\end{corollary}
\noindent
The following result we recall is quite general and stands for finite fractal structures. 
\begin{lemma} (c.f.~\cite[Proposition 3.5 (3)]{DIM4})\label{lema:1}
Let $\ef$ be a finite fractal structure on a metric space $(X,\rho)$, $F$ be a subset of $X$, and assume that $\deltafn\to 0$. Then $\dih(F)\leq\cuatro(F)\leq \tres(F)$.
\end{lemma}

\begin{corollary}[$\IFS$] \label{cor:2}
$\dih(\K) \leq \cuatro(\K)\leq \tres(\K)=\alpha$.
\end{corollary}

\begin{proof}
It follows as a consequence of Lemma \ref{lema:1}. In fact, recall that the natural fractal structure which any $\IFS-$attractor can be endowed with is finite. Further, such a fractal structure also satisfies that $\dm(\K,\Gamma_n)=\dm(\Gamma_n)\to 0$, since the sequence of diameters $\{\dm(\Gamma_n)\}_{n\in \N}$ decreases geometrically in the case of self-similar sets (c.f.~Remark~\ref{obs:1}). Finally, Theorem \ref{teo:OSC-H=3} gives $\tres(\K)=\alpha$.\newline
\end{proof}
The following result provides a Moran's type theorem (under the $\OSC$) involving fractal dimension IV as a consequence of previous corollaries. 
It can be understood as an extension of Corollary~\ref{cor:H=3=B}. 
\begin{theorem}[$\EIFS$]\label{teo:osc->3=4=H} 
$\OSC\Rightarrow \dih(\K)=\cuatro(\K)=\tres(\K)=\alpha$.
\end{theorem}

\begin{proof}
Notice that $\dih(\K)\leq \cuatro(\K)\leq \tres(\K)=\alpha$ by Corollary \ref{cor:2}.
Corollary \ref{cor:H=3=B} gives the result.
\end{proof}

To conclude this section, we recall two key results explored by Schief (c.f.~\cite{Schief1994,Schief1996}). They provide sufficient conditions to reach Moran's type theorems in both contexts: complete metric spaces and Euclidean $\IFS-$attractors. Such conditions consist of appropriate separation properties for $\IFS-$attractors.
\begin{theorem}\label{teo:schief}~
\begin{enumerate}
\item [$(\EIFS)$] $\SOSC\Leftrightarrow \OSC\Leftrightarrow \mathcal{H}_{\rm H}^{\alpha}(\K)>0 \Rightarrow \h(\K)=\alpha$.
\item [$(\IFS)$] $\mathcal{H}_{\rm H}^{\alpha}(\K)>0 \Rightarrow \SOSC\Rightarrow \h(\K)=\alpha$.
\end{enumerate}
\end{theorem}
Interestingly, Mattila provided the following counterexample which allows to justify that Theorem~\ref{teo:schief} is best possible.

\begin{matti*}[c.f.~\cite{Schief1994,Schief1996}]
Let $\F=\{f_1,f_2,f_3\}$ be an $\EIFS$ on $\R^2$, where $f_i(x)=x_i+\tfrac13(x-x_i)$ with $x_1=(0,0),x_2=(1,0),$ and $x_3=(\tfrac12,\tfrac{1}{2\sqrt{3}})$. The attractor of $\F$, $\K$, is a nonconnected Sierpi\'nski gasket in the plane which satisfies the $\SOSC$ and has similarity dimension $\alpha=1$. Thus, almost all Lebesgue projections of $\K$ on $1-$dimensional subspaces of $\R^2$ (that are self-similar set themselves), have Hausdorff dimension $1$ (due to Marstrand's Projection Theorem, c.f.~\cite[Projection theorem 6.1]{FAL90} and \cite{Kaufman1968,Marstrand1954}) but zero $\mathcal{H}_{\rm H}^1$ measure. 
\end{matti*}
From Theorem~\ref{teo:schief}, it holds that the $\OSC$ provides a sufficient (though not necessary) condition to get the equality between the similarity and the Hausdorff dimensions of Euclidean $\IFS-$attractors. It is worth mentioning that the implication $\mathcal{H}_{\rm H}^{\alpha}(\K)>0\Rightarrow \SOSC$ was contributed by Schief (c.f.~\cite[Theorem 2.1]{Schief1994}). This guarantees the equivalence among $\OSC,\SOSC$, and the $\alpha-$dimensional Hausdorff measure of $\K$. 
However, in the case of complete metric spaces, the $\OSC$ no longer leads to a Moran's type theorem (c.f.~\cite[Example 3.1]{Schief1996}). Thus, it must be replaced by the $\SOSC$ for that purpose (c.f.~\cite[Theorem 2.6]{Schief1996}). In addition, \cite[Example 3.2]{Schief1996} highlights that the $\SOSC$ does not suffice to guarantee a positive Hausdorff measure in the general case. Also, from Mattila's Counterexample it holds that the implication $\h(\K)=\alpha\Rightarrow \SOSC$ does not stand, in general.

\section{Irreducible fractal structures}

In this section, we introduce the concept of an irreducible fractal structure and characterize it in terms of the similarity dimension of an $\IFS-$attrac\-tor and its fractal dimensions III and IV, as well.

\begin{definition}
Let $\Gamma$ be a covering of $X$. We say that $\Gamma$ is irreducible provided that it has no proper subcoverings (c.f.~\cite[Problem 20D]{Willard1970}). By an irreducible fractal structure, we shall understand a fractal structure whose levels are irreducible coverings.
\end{definition}

First, we provide a sufficient condition leading to irreducible fractal structures. It consists of the equality between both fractal dimensions III and IV.

\begin{prop}[IFS]\label{prop:4=3->WOSC}
$\cuatro(\K)=\tres(\K)=\alpha\Rightarrow \ef$ irreducible.
\end{prop}

\begin{proof}
Assume, by the contrary, that $\ef$ is not irreducible. Then there exist $m \in \N$ and $\mathbf{i} \in \Sigma^m$ such that $\Gamma_m \setminus \{f_{\mathbf{i}}(\K)\}$ is a covering of $\K$.
Let $J=\Sigma^m \setminus \{\mathbf{i}\}$. Since $\sum_{\mathbf{j} \in \Sigma^m} c_{\mathbf{j}}^{\alpha}=1$ with $\alpha=\tres(\K)$, then $\sum_{\mathbf{j} \in J} c_{\mathbf{j}}^{\alpha}<1$. Thus, if $t$ is the unique solution of the equation $\sum_{\mathbf{j} \in J} c_{\mathbf{j}}^t=1$, then $t<\alpha$ since the function $s\mapsto \sum_{i=1}^kc_i^s$ is scrictly decreasing in $s$ (c.f.~\cite[Convention (2)]{HUT81}).
On the other hand, 
$\mathcal{H}_{n,4}^t(\K)=\mathcal{H}_{1,4}^t(\K)$ 
for each $n \in \N$, since any element $f_\mathbf{i}(\K) \in \Gamma_k$ with $\mathbf{i} \in \Sigma^k$ for some $k\in \N$, can be replaced by $\{f_{\mathbf{j}}(\K):\mathbf{j} \in \mathbf{i}J\}$, where $\dm(f_{\mathbf{i}}(\K) )^t=\sum_{\mathbf{j} \in \mathbf{i}J } \dm(f_{\mathbf{j}}(\K))^t$. Letting $n\to \infty$, it holds that $\mathcal{H}_4^{t}(\K)=\mathcal{H}_{1,4}^{t}(\K)<\infty$.
Accordingly, $\cuatro(\K) \leq t<\alpha$, a contradiction. 
\end{proof}

The following result stands from Proposition \ref{prop:4=3->WOSC}. It provides another sufficient condition to get irreducible fractal structures involving both fractal dimension III and Hausdorff dimension.  

\begin{corollary}[IFS]\label{cor:csirreducible}
$\dih(\K)=\tres(\K)=\alpha\Rightarrow \ef$ irreducible.
\end{corollary}

\begin{proof}
Corollary \ref{cor:2} gives $\dih(\K) \leq \cuatro(\K) \leq \tres(\K)=\alpha$.
Then $\dih(\K)=\alpha$ implies $\cuatro(\K)=\tres(\K)=\alpha$. Hence, $\ef$ is irreducible by Proposition~\ref{prop:4=3->WOSC}.
\end{proof}



Interestingly, the reciprocal of Proposition \ref{prop:4=3->WOSC} can also be stated.


\begin{prop}[IFS]\label{prop:irr->H4>0}
$\ef$ irreducible $\Rightarrow \mathcal{H}_4^{\alpha}(\K)>0$.
\end{prop}


\begin{proof}

Let us assume, by the contrary, that 
$\mathcal{H}_4^{\alpha}(\K)=0$.
Observe that any element $f_\mathbf{i}(\K) \in \Gamma_k$ with $\mathbf{i} \in \Sigma^k$ for some $k\in \N$, can be replaced by $\{f_{\mathbf{i}j}(\K):j \in \Sigma\}$, where $\dm(f_{\mathbf{i}}(\K))^{\alpha}=\sum_{j \in \Sigma} \dm(f_{\mathbf{i}j}(\K))^{\alpha}$, as many times as needed. This leads to $\mathcal{H}_{n,4}^{\alpha}(\K)=\mathcal{H}_{1,4}^{\alpha}(\K)$, and letting $n\to \infty$, we have $0=\mathcal{H}_4^{\alpha}(\K)=\mathcal{H}_{1,4}^{\alpha}(\K)$.    
On the other hand, let $\varepsilon=\frac{1}{2}\sum_{i \in \Sigma} \dm(f_{i}(\K))^{\alpha}$ and $\mathcal{A}$ be a finite covering of $\K$ by elements of $\cup_{n\in \N}\Gamma_n$, where $\sum_{A \in \mathcal{A}} \dm(A)^{\alpha}<\varepsilon$. In addition, let $n$ be the greatest level containing (at least) one element of $\mathcal{A}$. Thus, for all $m\leq n$ and each $f_{\mathbf{i}}(\K) \in \mathcal{A} \cap \Gamma_m$, we can replace $f_{\mathbf{i}}(\K)$ by $\{f_{\mathbf{j}}(\K):\mathbf{j} \in \Sigma^n, \mathbf{i} \sqsubseteq \mathbf{j} \}$, where $\dm(f_{\mathbf{i}}(\K) )^{\alpha}=\sum_{\mathbf{j} \in \Sigma^n, \mathbf{i} \sqsubseteq \mathbf{j} } \dm(f_{\mathbf{j}}(\K))^{\alpha}$. Hence, we can construct a covering $\mathcal{A}'\subseteq \Gamma_n$ from $\mathcal{A}$ such that $\sum_{A \in \mathcal{A}'} \dm(A)^{\alpha}=\sum_{A \in \mathcal{A}} \dm(A)^{\alpha}<\varepsilon$. Since $\Gamma_n$ is an irreducible covering and $\mathcal{A}'$ is a subcovering of $\Gamma_n$, then $\mathcal{A}'=\Gamma_n$ but $\sum_{A \in \Gamma_n} \dm(A)^{\alpha}=\sum_{i\in \Sigma}\dm(f_i(\mathcal{K}))^{\alpha}=2\varepsilon$, a contradiction.
\end{proof}

\begin{corollary}[IFS]\label{cor:irr->4=3}
$\ef$ irreducible $\Rightarrow \cuatro(\K)=\tres(\K)=\alpha$.
\end{corollary}

\begin{proof}
First, notice that $\cuatro(\K)\leq \tres(\K)=\alpha$ due to Corollary \ref{cor:2}. In addition, Proposition \ref{prop:irr->H4>0} gives $\mathcal{H}_4^{\alpha}(\K)>0$, so $\alpha\leq \cuatro(\K)$. 
\end{proof}
It is worth pointing out that Corollary \ref{cor:irr->4=3} can be understood as a Moran's type theorem for fractal dimensions III and IV. Interestingly,
both Proposition~\ref{prop:4=3->WOSC} and Corollary~\ref{cor:irr->4=3} lead to one of the key results we state in this paper.

\begin{theorem}[IFS]\label{teo:irreducible<-->4=3}
$\ef$ irreducible $\Leftrightarrow \cuatro(\K)=\tres(\K)=\alpha$.
\end{theorem}
In this way, if $\ef$ is the natural fractal structure on $\K$ as a self-similar set (c.f.~Definition~\ref{def:nfsifs}), then the condition $\cuatro(\K)=\tres(\K)=\alpha$ is equivalent to $\ef$ being irreducible. Thus, it allows to characterize irreducible fractal structures in terms of the equality among the similarity dimension of an $\IFS-$attractor and its fractal dimensions III and IV. 

\section{The level separation property for complete metric spaces}


In this section, we introduce a novel separation condition for each level of the natural fractal structure $\ef$ that any $\IFS-$attractor can be endowed with (c.f.~Definition~\ref{def:nfsifs}). We shall prove that such a separation property is equivalent to $\ef$ being irreducible. Moreover, we show additional equivalences between that separation property and other conditions to describe the structure of self-similar sets, e.g., $\ef$ being a tiling.

\begin{definition}
We shall understand that $\F$ satisfies the so-called level separation property ($\LSP$) if the two following conditions hold for each level of $\ef$:
\begin{equation}
\tag{$\LSPuno$}
\accentset{\circ}{A}\cap \accentset{\circ}{B}=\emptyset, \text{ for all } A,B\in \Gamma_n: A \not=B. 
\end{equation}
\begin{equation}
\tag{$\LSPdos$}
\accentset{\circ}{A} \not=\emptyset, \text{ for each } A \in \Gamma_n,
\end{equation}
where the interiors have been considered in $\K$.
\end{definition}
Thus, $\ef$ is under the $\LSP$ if for each level of the natural fractal structure on $\K$, it holds that the interiors of any two (different) elements are pairwise disjoint with their interiors being nonempty. It is worth pointing out that the $\LSP$ does not depend on an external open set, unlike the $\OSC$.
First, we prove that the $\LSP$ implies $\ef$ being irreducible.

\begin{prop}[IFS]\label{prop:WOSC->irreducible}
$\WSP\Rightarrow \ef$ irreducible.
\end{prop}

\begin{proof}
Assume that $\ef$ is not irreducible. 
Then there exist $n \in \N$ and $A\in \Gamma_n$ such that $A\subseteq \cup\{B\in \Gamma_n:B\neq A\}$. Let us denote $\beta=\{B\in \Gamma_n:B\neq A\}$. 
By $\LSPdos$, there exists $x \in A^{\circ}$. Since $\ef$ is starbase, there exists $m \geq n$ with $\St(x,\Gamma_m) \subseteq A$. Since $x \in A \subseteq \bigcup_{B \in \beta} B$, there exists $B \in \Gamma_n$ with $B \not=A$ and $x \in B$. By $\LSPuno$, we have $\inte{A}\cap \inte{B}=\emptyset$. Then there exists $C \in \Gamma_m$ such that $x \in C \subseteq B$, and hence $\inte{C} \subseteq \inte{B}$. On the other hand, $C \subseteq \St(x,\Gamma_m) \subseteq A$. Therefore $\inte{C} \subseteq \inte{B} \cap \inte{A}=\emptyset$, which is a contradiction with $\LSPdos$. We conclude that $\ef$ is irreducible.

\end{proof}

The reciprocal of Proposition \ref{prop:WOSC->irreducible} is also true.

\begin{prop}[IFS]\label{prop:irreducible->WOSC}
$\ef$ irreducible $\Rightarrow \WSP$.
\end{prop}

\begin{proof}
Let $\ef$ be an irreducible fractal structure. 
\begin{itemize}
\item First, we check that $\WSP2$ holds. To deal with, let $n \in \N$ and $A \in \Gamma_n$. Since $\Gamma_n$ is an irreducible covering, then it follows that $\cup_{B \in \Gamma_n, B \not=A} B \not=\mathcal{K}$. Thus, $\mathcal{K}\setminus \cup_{B \in \Gamma_n, B \not=A} B$ is a nonempty open set contained in $A$ and hence, $\accentset{\circ}{A}\neq\emptyset$.
\item Next step is to prove $\WSP1$. Assume, by the contrary, that $\F$ does not satisfy $\LSPuno$. Then there exist a level $n\in \N$, $x \in X$ and elements $A,B\in \Gamma_n:A\neq B$ such that $x\in \accentset{\circ}{A}\cap \accentset{\circ}{B}$. Moreover, since $\ef$ is starbase, then there exists $m>n$ such that $\St(x,\Gamma_m)\subseteq \accentset{\circ}{A}\cap \accentset{\circ}{B}$. We can write $A=f_{\mathbf{i}}(\K), B=f_{\mathbf{j}}(\K):\mathbf{i},\mathbf{j}\in \Sigma^n, \mathbf{i}\neq \mathbf{j}$. Let $C\in \Gamma_m$ such that $C=f_{\mathbf{l}}(\K)$, where $\mathbf{i}\sqsubseteq \mathbf{l}$ and $x \in C$, so $\mathbf{j}\not\sqsubseteq \mathbf{l}$, then $C\subseteq \St(x,\Gamma_m)\subseteq \accentset{\circ}{A}\cap \accentset{\circ}{B}\subseteq A\cap B$. Since $C \subseteq B=\bigcup_{k \in \Sigma^{m-n}}f_{\mathbf{j}\mathbf{k}}(\K)$, then $\Gamma_m \setminus \{C\}$ is a cover of $\K$, so $\Gamma_n$ is not irreducible, a contradiction.
\end{itemize}
\end{proof}
 
Both Propositions~\ref{prop:WOSC->irreducible} and \ref{prop:irreducible->WOSC} yield the following equivalence between the $\WSP$ and $\ef$ being an irreducible fractal structure.

\begin{theorem}[IFS]\label{prop:WOSC<-->irr}
$\WSP\Leftrightarrow \ef$ irreducible.
\end{theorem}



An additional equivalence regarding the $\LSP$ 
is shown next.

\begin{prop}[IFS]\label{prop: irr-->nuevo}
$\ef$ irreducible $\Rightarrow$ $\F$ satisfies $\WSP2$ and 
$A_{\mathbf{i}} \subseteq A_{\mathbf{j}}$ implies $\mathbf{j} \sqsubseteq \mathbf{i}$.
\end{prop}

\begin{proof}
Suppose that $\ef$ is irreducible. Then Proposition \ref{prop:irreducible->WOSC} guarantees $\WSP2$, so $\accentset{\circ}{A} \not=\emptyset$, for each $A \in \Gamma_n$ and all $n \in \N$.
In addition, let $\mathbf{i} \in \Sigma^n, \mathbf{j} \in \Sigma^m$ be such that $A_{\mathbf{i}}\subseteq A_{\mathbf{j}}$. Then one of the two following cases occurs.
\begin{enumerate}[(i)]
\item Assume that $m>n$. Then there exists $\mathbf{k} \in \Sigma^m$ such that $\mathbf{i} \sqsubseteq \mathbf{k}$ and $A_{\mathbf{k}} \subseteq A_{\mathbf{i}}$. Thus, $A_{\mathbf{k}} \subseteq A_{\mathbf{j}}$. Further, since $\Gamma_m$ is irreducible, then $\mathbf{k}=\mathbf{j}$. Hence, $\mathbf{i} \sqsubseteq \mathbf{j}$ and $A_{\mathbf{j}} \subseteq A_{\mathbf{i}}$. Therefore, $\mathbf{i} \sqsubseteq \mathbf{j}$ and $A_{\mathbf{j}} = A_{\mathbf{i}}$. But this is only possible if $n=m$, a contradiction.
\item Suppose that $m \leq n$. Then $A_{\mathbf{i}}\subseteq A_{\mathbf{j}}=\cup_{\mathbf{k}\in \Sigma^n,\, \mathbf{j} \sqsubseteq \mathbf{k}} A_{\mathbf{k}}$. Since $\Gamma_n$ is irreducible, then $\mathbf{i} = \mathbf{k}$ for some $\mathbf{k} \in \Sigma^n$ with  $\mathbf{j} \sqsubseteq \mathbf{k}$, and hence, $\mathbf{j} \sqsubseteq \mathbf{i}$.
\end{enumerate}
\end{proof}

\begin{prop}[IFS] \label{prop: nuevo-->irr}
If $\F$ is under $\WSP2$ and 
the condition 
$A_{\mathbf{i}} \subseteq A_{\mathbf{j}}$ implies $\mathbf{j} \sqsubseteq \mathbf{i}$, then $\ef$ is irreducible.
\end{prop}

\begin{proof}
Suppose that 
$\F$ is under $\LSPdos$ and the condition $A_{\mathbf{i}} \subseteq A_{\mathbf{j}}$ implies $\mathbf{j} \sqsubseteq \mathbf{i}$.
Assume, in addition, that $\ef$ is not irreducible. Then there exists $n \in \N$ and $\mathbf{i} \in \Sigma^n$ such that $A_{\mathbf{i}} \subseteq \cup_{\mathbf{j} \in \Sigma^n,\ \mathbf{j} \not=\mathbf{i}} A_{\mathbf{j}}$.

By $\LSPdos$, there exists $x \in \inte{A_{\mathbf{i}}}$. Since $\ef$ is starbase, there exists $m>n$ such that $\St(x,\Gamma_m) \subseteq A_{\mathbf{i}}$. On the other hand, since $x \in A_{\mathbf{i}} \subseteq \cup_{\mathbf{j} \in \Sigma^n,\ \mathbf{j} \not=\mathbf{i}} A_{\mathbf{j}}$, there exists $\mathbf{j} \in \Sigma^n$  with $\mathbf{j} \not=\mathbf{i}$ and $x \in A_{\mathbf{j}}$. Let $\mathbf{k} \in \Sigma^m$ be such that $\mathbf{j} \sqsubseteq \mathbf{k}$ and $x \in A_{\mathbf{k}}$.

It follows that $x \in A_{\mathbf{k}} \subseteq \St(x,\Gamma_m) \subseteq A_{\mathbf{i}}$, so $A_{\mathbf{k}} \subseteq A_{\mathbf{i}}$ and by hypothesis $\mathbf{i} \sqsubseteq \mathbf{k}$, which is a contradiction, since $\mathbf{j} \sqsubseteq \mathbf{k}$ and $\mathbf{j}, \mathbf{i} \in \Sigma^n$ with $\mathbf{i} \not= \mathbf{j}$.

\end{proof}

Both Propositions \ref{prop: irr-->nuevo} and \ref{prop: nuevo-->irr} give the following equivalence.

\begin{theorem}[IFS] \label{prop: irr<-->nuevo}
$\ef$ irreducible $\Leftrightarrow$ $\F$ is under $\WSP2$ and 
$A_{\mathbf{i}} \subseteq A_{\mathbf{j}}$ implies $\mathbf{j} \sqsubseteq \mathbf{i}$.
\end{theorem}

We also provide a further characterization for $\ef$ being irreducible. To deal with, first we recall when a covering is said to be a tiling.

\begin{definition}(c.f.~\cite[Section 2]{MR1881446})\label{def:tiling}
Let $\Gamma$ be a covering of $X$. We shall understand that $\Gamma$ is a tiling provided that all the elements of $\Gamma$ have disjoint interiors and are regularly closed, i.e., $\overline{A^{\circ}}=A$ for each $A\in \Gamma$. In particular, a fractal structure $\ef$ is called a tiling if each level $\Gamma_n$ of $\ef$ is a tiling itself.
\end{definition}
From Definition~\ref{def:tiling}, it is clear that $\ef$ tiling $\Rightarrow \LSP$. The following result deals with its reciprocal.   


\begin{theorem}[IFS]\label{teo:irreducible<-->tiling}
$\LSP$ $\Leftrightarrow$ $\ef$ tiling.
\end{theorem}

\begin{proof}~
\begin{enumerate}
\item [$(\Leftarrow)$] Obvious.
\item [$(\Rightarrow)$] 
Suppose that there exist $n \in \N$, $A \in \Gamma_n$ and $x \in A \setminus \overline{A^{\circ}}$. Since $\ef$ is starbase, then there exists $m \in \N$ with $m \geq n$ and such that $\St(x,\Gamma_m) \cap A^{\circ}=\emptyset$. Let $B \in \Gamma_m$ be such that $x \in B \subseteq A$. Then $B \subseteq A \setminus A^{\circ}=\Fr(A)$, and hence $B^{\circ}=\emptyset$, a contradiction with $\LSPdos$. Therefore, $A=\overline{A^{\circ}}$, i.e., $A$ is regularly closed. 
\end{enumerate}
\end{proof}

\begin{corollary}[IFS]\label{cor:irr<-->tiling}
$\ef$ irreducible $\Leftrightarrow$ $\ef$ tiling.
\end{corollary}

All the results obtained so far are collected in the next result.

\begin{theorem}[IFS]\label{teo:all}
The following statements are equivalent: 
\begin{enumerate}[(i)]
\item $\ef$ irreducible.\label{item:1}
\item $\cuatro(\K)=\tres(\K)=\alpha$.\label{item:2}
\item $\WSP$.\label{item:3}
\item $\WSP2$ and $A_{\mathbf{i}} \subseteq A_{\mathbf{j}}$ implies $\mathbf{j} \sqsubseteq \mathbf{i}$.\label{item:4}
\item $\ef$ tiling.\label{item:5}
\item $\mathcal{H}_4^{\alpha}(\K)>0$.\label{item:6}
\end{enumerate}
\end{theorem}

\begin{proof}
Just notice that $(\ref{item:1})\Leftrightarrow (\ref{item:2})$ is due to Theorem~\ref{teo:irreducible<-->4=3}, $(\ref{item:1})\Leftrightarrow (\ref{item:3})$ holds by Theorem~\ref{prop:WOSC<-->irr}, $(\ref{item:1})\Leftrightarrow (\ref{item:4})$ follows from Theorem~\ref{prop: irr<-->nuevo}, $(\ref{item:1})\Rightarrow (\ref{item:6})$ follows from Proposition \ref{prop:irr->H4>0}, and $(\ref{item:3})\Leftrightarrow (\ref{item:5})$ is a consequence of Corollary~\ref{cor:irr<-->tiling}. 
Thus, we shall prove that $(\ref{item:6})\Rightarrow (\ref{item:2})$ to conclude the proof. In fact, $\mathcal{H}_4^{\alpha}(\K)>0$ leads to $\alpha\leq \cuatro(\K)$. Hence, $\cuatro(\K)=\tres(\K)=\alpha$ due to Corollary~\ref{cor:2}. 
\end{proof}

\section{The weak separation condition for $\IFS-$attractors}\label{sec:wsc}

Theorem~\ref{teo:all} allows us to introduce the so-called weak open set condition for $\IFS-$attractors.

\begin{definition}\label{def:wosc}
We shall understand that $\F$ is under the weak separation condition ($\WOSC$) if any of the equivalent statements provided in Theorem~\ref{teo:all} holds.
\end{definition}



Unlike the $\OSC$, the $\WOSC$ does not depend on any external open set (c.f.~ Definition~\ref{def:wosc}).  
It is also worth pointing out that the $\WOSC$ becomes a separation property for $\IFS-$attractors weaker than the $\OSC$. Additionally, it allows the calculation of their similarity dimensions. Next, we shall justify these facts. 

\begin{corollary}[$\EIFS$]\label{cor:OSC->WOSC}
$\OSC\equiv\SOSC\Rightarrow \ef$ irreducible.
\end{corollary}

\begin{proof}
By Moran's Theorem, $\OSC$ implies $\dih(\K)=\alpha$. Hence, $\ef$ is irreducible just by applying both Theorems~\ref{teo:osc->3=4=H} and \ref{teo:irreducible<-->4=3}.
\end{proof}
However, 
Corollary~\ref{cor:OSC->WOSC} cannot be inverted, in general.
\begin{remark}[$\EIFS$]\label{obs:oscno->wosc}
$\ef$ irreducible $\centernot{\Rightarrow}\OSC$.
\end{remark}

\begin{proof}
Consider Mattila's Counterexample. In fact, the $\IFS-$attractor provided therein has zero $\mathcal{H}_{\rm H}^1$ measure. Thus, it does not satisfy the $\OSC$. In addition, since the Hausdorff dimension of such an $\IFS-$attractor equals its similarity dimension, then Corollary~\ref{cor:csirreducible} leads to $\ef$ irreducible. 
\end{proof}
The next result stands in the general case as a consequence of both Corollary~\ref{cor:OSC->WOSC} and Remark~\ref{obs:oscno->wosc}. It states that the $\WOSC$ becomes necessary to fulfill the $\SOSC$ in the general case. 
\begin{corollary}[$\IFS$]\label{cor:74}
$\SOSC\Rightarrow \WOSC$ and the reciprocal is not true, in general. 
\end{corollary}
Thus, the $\WOSC$ remains necessary for the $\SOSC$ in the general case.
In short, as well as the $\SOSC$ is sufficient to achieve the equality $\h(\K)=\alpha$ in the context of complete metric spaces, the following Moran's type theorem holds for both fractal dimensions III and IV provided that $\F$ is under the $\WOSC$.
\begin{theorem}[$\IFS$]\label{teo:wosc<->4=3=s}
$\WOSC\Leftrightarrow \cuatro(\K)=\tres(\K)=\alpha$.
\end{theorem}
It is worth noting that Theorem~\ref{teo:wosc<->4=3=s} does not guarantee the identity $\h(\K)=\alpha$, unlike Schief's Theorem~\ref{teo:schief} ($\IFS$ case). In return, it provides an equivalence between the $\WOSC$, 
a separation property for $\IFS-$attractors weaker than the $\OSC$, and the identity $\cuatro(\K)=\tres(\K)=\alpha$. 
Both results allow the calculation of the self-similarity dimension of $\K$. 

\section{Sufficient conditions for irreducible fractal structures}
Along this section, we shall explore several properties on a fractal structure leading to the $\WOSC$. 
To tackle with, recall that the order of a fractal structure $\ef$ is defined as $\ord(\ef)=\sup\{\ord(\Gamma_n):n\in \N\}$, where $\ord(\Gamma_n)$ is the order of level $n$ as a covering, i.e., $\ord(\Gamma_n)=\sup\{\ord(x,\Gamma_n):x\in X\}$, and $\ord(x,\Gamma_n)=\cd(\{A\in \Gamma_n:x\in A\})-1$ (c.f.~\cite{asg03a}).




The first result we provide in that direction deals with the general case of IFS-attractors on complete metric spaces.

\begin{prop}[$\IFS$]\label{prop:k-cond}
Let us consider the following list of statemens:
\begin{enumerate}[(i)]
\item There exists $k \in \N$ such that each element in $\Gamma_n$ intersects at most to $k$ other elements in that level, for each $n \in \N$.\label{item:12}
\item $\ord(\ef)<\infty$.\label{item:13}
\item $\ef$ irreducible.\label{item:14}
\end{enumerate} 
Then $(\ref{item:12})\Rightarrow (\ref{item:13})\Rightarrow (\ref{item:14})$.
\end{prop}

\begin{proof}
The implication 
$(\ref{item:12})\Rightarrow (\ref{item:13})$ is clear. Thus, we shall be focused on how to tackle with $(\ref{item:13})\Rightarrow (\ref{item:14})$.
Let us assume, by the contrary, that $\ord(\ef)<\infty$ with $\ef$ not being irreducible. Then there exist $m \in \N$ and $\mathbf{i} \in \Sigma^m$ such that $A_{\mathbf{i}} \subseteq \cup_{\mathbf{j}\in \Sigma^m,\, \mathbf{j}\not=\mathbf{i}} A_{\mathbf{j}}$. Since $\ord(\ef)$ is finite, then there exist $k \in \N$, $x \in \K$, and $n \in \N$, such that $x$ belongs exactly to $k$ elements of $\Gamma_n$, and any other point in $\K$ belongs at most to $k$ elements of $\Gamma_o$, for any $o \in \N$. Hence, $f_{\mathbf{i}}(x)$ belongs to $k$ elements of $\Gamma_{n+m}$ of the form $A_{\mathbf{j}}$ with $\mathbf{j} \in \Sigma^{n+m}$ and $\mathbf{i} \sqsubseteq \mathbf{j}$, and since $f_{\mathbf{i}}(x) \in A_{\mathbf{i}} \subseteq \cup_{\mathbf{j} \in \Sigma^m,\, \mathbf{j} \not=\mathbf{i}} A_{\mathbf{j}}$, then it holds that $f_{\mathbf{i}}(x)$ belongs to some $A_{\mathbf{j}}$ with $\mathbf{j} \in \Sigma^{n+m}$ and $\mathbf{i} \not\sqsubseteq \mathbf{j}$. Thus, $f_{\mathbf{i}}(x)$ belongs at least to $k+1$ elements of $\Gamma_{n+m}$, a contradiction.
\end{proof}


Similarly to Proposition~\ref{prop:k-cond}, the following result stands in the Euclidean case. 

\begin{prop}[$\EIFS$]\label{prop:k-cond2}
Consider the following list of statemens:
\begin{enumerate}[(i)]
\item $\OSC$.\label{item:02}
\item $\ord(\ef)<\infty$.\label{item:03}
\item $\ef$ irreducible.\label{item:04}
\end{enumerate} 
Then 
$(\ref{item:02})\Rightarrow (\ref{item:03})\Rightarrow (\ref{item:04})$.
\end{prop}

\begin{proof}
The implication $(\ref{item:03})\Rightarrow (\ref{item:04})$ has been proved in Proposition~\ref{prop:k-cond}. 
To deal with $(\ref{item:02})\Rightarrow (\ref{item:03})$, observe that the $\OSC$ implies that there exists an integer $N$ such that $A$ intersects to $\leq N$ incomparable pieces $A_{\mathbf{j}}\in \cup_{k\in \N} \Gamma_{k}$ with $\dm(A)\leq \dm(A_{\mathbf{j}})$, for all $A\in \Gamma_n$ and all $n\in \N$. Let $x$ be a point, $n\in \N$, and assume that $x$ belongs to $>N+1$ pieces of $\Gamma_n$. Let $A$ be the lowest diameter piece of $\Gamma_n$ containing $x$. Thus, $A$ intersects to $>N$ elements of $\Gamma_n$ that are incomparable since all of them lie in the same level of $\ef$, a contradiction.   
\end{proof}
The result above states that irreducible fractal structures becomes necessary to fulfill the $\OSC$ in the Euclidean case.

\section{Regarding the size of the overlaps}\label{sec:size}

In a previous work due to Bandt and Rao (c.f.~\cite{Bandt2007}), they proved the following result regarding the size of the overlaps.

\begin{theorem}[c.f.~\cite{Bandt2007}, Theorem 2]\label{teo:bandt-rao}
Let $\K$ be a connected self-similar set in $\R^2$. If $\K_i\cap \K_j$ is a finite set for $i\neq j$, then the $\OSC$ holds.
\end{theorem}
Proposition~\ref{prop:k-cond2} yields the following consequence from Bandt-Rao's theorem.
\begin{corollary}[$\EIFS$]\label{cor:bandt-rao}
Let $\K$ be a connected self-similar set in $\R^2$ endowed with its natural fractal structure. If $\K_i\cap \K_j$ is a finite set for $i\neq j$, then the $\WOSC$ holds.
\end{corollary}
Thus, similarly to the problem appeared in Section 1 of \cite{Bandt2007} regarding the $\OSC$, next we pose the analogous concerning the $\WOSC$:
\begin{problem*}
Will the $\WOSC$ be true if all the overlaps are finite sets?
\end{problem*}

In this occasion, though, an affirmative result can be proved in the general case.

\begin{theorem}[$\IFS$]\label{cor:fsize}
If $A\cap B$ is a finite set for all $A,B\in \Gamma_n$ with $A \not=B$ and each $n\in \N$, then the $\LSP$ stands. Hence, the $\WOSC$ holds. 
\end{theorem}

\begin{proof}
Observe that $\K\setminus\cup\{B:A\neq B\}=A\setminus\cup\{B:A\neq B\}=A\setminus F$ with $F$ being a finite set. Thus, $\K\setminus\cup\{B:A\neq B\}$ is a nonempty open set (on $\K$) contained in $A$, so $\LSPdos$ holds. On the other hand, 
since $A^{\circ}\cap B^{\circ}\subseteq (A\cap B)^{\circ}$ and $A\cap B$ is finite for all $A,B\in \Gamma_n$, then $A^{\circ}\cap B^{\circ}=\emptyset$. Hence, $\LSPuno$ stands.    
\end{proof}

\begin{corollary}[$\IFS$]\label{cor:fsize}
If $A\cap B$ is a finite set for all $A,B\in \Gamma_n$ with $A \not=B$ and all $n\in \N$, then the $\WOSC$ holds.
\end{corollary}

It is worth pointing out that Theorem~\ref{cor:fsize} stands in the general case of IFS-attractors on a complete metric space, unlike Corollary~\ref{cor:bandt-rao}, which only stands in the Euclidean plane. In addition, Theorem~\ref{cor:fsize} does not require the self-similar set to be connected.

\section{Conclusion}\label{sec:conclusion}

In this section, we summarize all the results contributed along this paper.


\begin{theorem}

Consider the following statements:
\begin{enumerate}[(i)]
\item $\mathcal{H}_{\rm H}^{\alpha}(\K)>0$.\label{item:a0}
\item $\SOSC$.\label{item:a1}
\item $\OSC$.\label{item:a2}
\item $\h(\K)=\cuatro(\K)=\tres(\K)=\alpha$.\label{item:a3}
\item $\cuatro(\K)=\tres(\K)=\alpha$.\label{item:a4}
\item $\ef$ irreducible.\label{item:a5}
\item $\ef$ tiling.\label{item:a6}
\item $\mathcal{H}_4^{\alpha}(\K)>0$.\label{item:a7}
\end{enumerate}
The next chains of implications and equivalences hold in each case:
\begin{enumerate}
\item [$(\EIFS)$] $(\ref{item:a0})\Leftrightarrow(\ref{item:a1})\Leftrightarrow(\ref{item:a2})\Rightarrow
(\ref{item:a3})\Rightarrow (\ref{item:a4})\Leftrightarrow (\ref{item:a5})
\Leftrightarrow (\ref{item:a6})
\Leftrightarrow (\ref{item:a7})$.
\item [$(\IFS)$] $(\ref{item:a0})\Rightarrow(\ref{item:a1})\Rightarrow
(\ref{item:a3})\Rightarrow (\ref{item:a4})\Leftrightarrow (\ref{item:a5})
\Leftrightarrow (\ref{item:a6})
\Leftrightarrow (\ref{item:a7})$.
\end{enumerate}
\end{theorem}
In addition, the next theorem highlights the connections among separation conditions for $\IFS-$attractors.

\begin{theorem}\label{teo:resumen2}
Consider the following list of statements:
\begin{enumerate}[(i)]
\item $\SOSC$.\label{item:a11}
\item $\OSC$.\label{item:a12}
\item $\WOSC$.\label{item:a14}
\end{enumerate}
The next chains of implications and equivalences hold and are best possible:
\begin{enumerate}
\item [$(\EIFS)$] $(\ref{item:a11})\Leftrightarrow(\ref{item:a12})\Rightarrow (\ref{item:a14})$.
\item [$(\IFS)$] $(\ref{item:a11})\Rightarrow (\ref{item:a12})$ and $(\ref{item:a11})\Rightarrow (\ref{item:a14})$.
\end{enumerate}
\end{theorem}
A counterexample due to Schief (c.f.~\cite[Example 3.1]{Schief1996}), guarantees the existence of (non-Euclidean) IFS-attractors under the $\OSC$ that do not satisfy the $\SOSC$.


\begin{remark}
Regarding Theorem~\ref{teo:resumen2}, it is worth noting that the following implications are not true, in general:
\begin{enumerate}
\item [(EIFS)] $(\ref{item:a14})\Rightarrow(\ref{item:a12})$ (c.f.~Mattila's Counterexample).
\item [(IFS)] $(\ref{item:a12})\Rightarrow (\ref{item:a14})$  due to Schief's Counterexample. In fact, the $\WOSC$ does not hold there since $f_1(\K)=f_2(\K)$.
$(\ref{item:a14})\Rightarrow (\ref{item:a12})$ (c.f.~Mattila's Counterexample), $(\ref{item:a14})\Rightarrow (\ref{item:a11})$ (c.f. either Corollary \ref{cor:74} or Mattila's Counterexample), and $(\ref{item:a12})\Rightarrow (\ref{item:a11})$ (c.f.~Schief's Counterexample).
\end{enumerate}
\end{remark}

To conclude this section, we provide two comparative theorems (one for each context, $\EIFS$ or $\IFS$) involving our results and those obtained by Schief.
\begin{theorem}[EIFS, comparative theorem in Euclidean case]\label{teo:comp1}~
\begin{enumerate}
\item [\textbf{Schief's:}] $\mathcal{H}_{\mathrm H}^{\alpha}(\K)>0\Leftrightarrow \OSC\Leftrightarrow \SOSC\Rightarrow \h(\K)=\alpha$.
\item [\textbf{Our's:}] $\WOSC\Leftrightarrow \mathcal{H}_4^{\alpha}(\K)>0\Leftrightarrow \cuatro(\K)=\alpha$.
\end{enumerate}
\end{theorem}

\begin{theorem}[IFS, comparative theorem for complete metric spaces]\label{teo:comp2}~
\begin{enumerate}
\item [\textbf{Schief's:}] $\mathcal{H}_{\mathrm H}^{\alpha}(\K)>0\Rightarrow \SOSC\Rightarrow \h(\K)=\alpha$.
\item [\textbf{Our's:}] $\WOSC\Leftrightarrow \mathcal{H}_4^{\alpha}(\K)>0\Leftrightarrow \cuatro(\K)=\alpha$.
\end{enumerate}
\end{theorem}
Both Theorems~\ref{teo:comp1} and \ref{teo:comp2} are best possible since no other implications are valid, in general. Thus, it holds that in the Euclidean case, the $\SOSC$ becomes sufficient but not necessary (c.f.~Mattila's Counterexample) 
to achieve the equality $\h(\K)=\alpha$. 
On the other hand, we have proved the equivalence among the $\WOSC$ and the conditions $\cuatro(\K)=\alpha$ and $\mathcal{H}_4^{\alpha}(\K)>0$ (that is calculated by finite coverings), as well. 
In the general case, though, the $\SOSC$ is only sufficient to get the equality between $\h(\K)$ and $\alpha$. Our chain of equivalences involving the $\WOSC$ still remains valid for complete metric spaces. We have also proved that the $\SOSC$ is stronger than the $\WOSC$ (once again, the Mattila's Counterexample works), a separation property for $\IFS-$attractors which does not depend on any external open set (unlike the $\SOSC$). Anyway, both statements in Theorem \ref{teo:comp2} (Schief's and our's) can be combined into the following summary result standing in the general case:

\begin{corollary}[IFS]\label{cor:85}
\[\mathcal{H}_{\mathrm H}^{\alpha}(\K)>0\Rightarrow \SOSC\Rightarrow \h(\K)=\alpha\Rightarrow
\WOSC\Leftrightarrow \mathcal{H}_4^{\alpha}(\K)>0\Leftrightarrow \cuatro(\K)=\alpha.
\]
\end{corollary}

Interestingly, Corollary~\ref{cor:85} highlights that the $\WOSC$ becomes necessary to reach the equality between the Hausdorff and the similarity dimensions of IFS-attractors. In other words, if the natural fractal structure which any IFS-attractor can be endowed with is not irreducible, then a Moran's type theorem cannot hold. 

To conclude this paper, we shall pose some natural questions still remaining open.

\begin{open}[$\EIFS/\IFS$]\label{open:1}
Is it true that $\WOSC\Rightarrow \h(\K)=\alpha$?
\end{open}

\begin{open}[$\EIFS/\IFS$]\label{open:2}
Is it true that $\h(\K)=\cuatro(\K)$?
\end{open}
It is worth noting that an affirmative response to Open question~\ref{open:2} would imply that Open question~\ref{open:1} is true.

\end{document}